\documentclass[sumlimits,intlimits,reqno,12pt]{amsart}
\usepackage{amsmath, amsfonts, amssymb,amsthm}
 \theoremstyle{plain}
\newtheorem{theorem}{Theorem}[section]
\newtheorem{lemma}[theorem]{Lemma}
\newtheorem{corollary}[theorem]{Corollary}
\newtheorem{proposition}[theorem]{Proposition}

\theoremstyle{definition}

\theoremstyle{remark}

\numberwithin{equation}{section}

\def\glc{generalized local cohomology }

\def\Ext{\mathrm{Ext}}
\def\Tor{\mathrm{Tor}}
\def\Hom{\mathrm{Hom}}
\def\Ker{\mathrm{Ker}}
\def\Coker{\mathrm{Coker}}
\def\Im{\mathrm{Im}}
\def\pd{\mathrm{pd}}
\def\Supp{\mathrm{Supp}}
\def\Ass{\mathrm{Ass}}
\def\dim{\mathrm{dim}}
\def\P{\mathfrak{p}}

\begin{document}

\title{GENERALIZED IDEAL TRANSFORMS}


\author{TRAN TUAN NAM}
\address{Department of Mathematics-Informatics, Ho Chi Minh University of  Pedagogy, Ho Chi Minh city, Viet Nam.}
\curraddr{}
\email{namtuantran@gmail.com}
\thanks{This research is funded by Vietnam National Foundation for
Science and Technology Development (NAFOSTED)}

\author{NGUYEN MINH TRI}
\address{Department of Natural Science Education, Dong Nai University, Dong Nai, Viet Nam.}
\curraddr{}
\email{triminhng@gmail.com}
\thanks{}

\date{}

\begin{abstract}
We study  basic properties of the generalized ideal transforms
$D_I(M, N)$   and   the set of associated primes of the modules
$R^iD_I(M,N).$
\end{abstract}

\dedicatory{}

\maketitle


\bigskip

\section{Introduction}\label{intro}

Throughout this paper, $R$ is a  Noetherian commutative ring with
non-zero identity and $I$ is an ideal of $R$. In \cite{broloc},
Brodmann defined  ideal transform $D_I(M)$ of an $R-$module $M$ with
respect to $I$ by $$D_I(M)=\mathop {\lim
}\limits_{\begin{subarray}{c}
   \longrightarrow  \\
  n
\end{subarray}}  \Hom_R(I^{n},M).$$Ideal transforms turn out to be
a powerful tool in various fields of commutative algebra and they
are  closed to local cohomology modules of Grothendieck.

In \cite{herkom}, Herzog introduced the definition of generalized
local cohomology modules which is an extension of local cohomology
modules of Grothendieck. The $i$-th generalized local cohomology
module of modules $M$ and $N$ with respect to $I$ was given as
$$H^i_I(M,N)=\mathop {\lim }\limits_{\begin{subarray}{c}
   \longrightarrow  \\
  n
\end{subarray}}  \Ext^i_R(M/I^{n}M,N).$$
A natural way, we have a   generalization of the ideal transform. In
\cite{divcof},  the generalized ideal transform functor with respect
to an ideal $I$ is defined by
$$D_I(M,-)=\mathop {\lim }\limits_{\begin{subarray}{c}
   \longrightarrow  \\
  n
\end{subarray}}  \Hom_R(I^{n}M,-).$$Also in \cite{divcof} they used the generalized
ideal  transforms  to study the cofiniteness of generalized local cohomology
modules. Let $R^iD_I(M,-)$ denote the $i$-th right derived functor
of $D_I(M,-).$  It is clear that $$R^iD_I(M,-)\cong \mathop {\lim
}\limits_{\begin{subarray}{c}
   \longrightarrow  \\
  n
\end{subarray}}  \Ext^i_R(I^{n}M,-)$$
for all $i\geq 0.$

The organization of our paper is as follows. In the next section we
study basic properties of the generalized ideal transform functor
$D_I(M,-)$ and its right derived functors $R^iD_I(M,-).$ The first
result is Theorem \ref{DNfree} which says that if $M$ is a finitely
generated $R$-module and $N$ is an $I$-torsion $R$-module, then
$R^iD_I(M,N)=0$ for all $i\geq 0.$ Next, Theorem \ref{DiHom} gives
us isomorphisms $D_I(\Hom_R(M,N))\cong D_I(M,N)$ and
$D_{aR}(M,N)\cong D_{aR}(M,N)_a.$ In Theorem \ref{HomFi} we see that
the module $\Hom_R(R/I, R^tD_I(M,N))$ is a finitely generated
$R$-module provided the modules $R^iD_I(M,N)$ are finitely generated
for all $i<t$. The section is closed by Theorem
\ref{T:tinhatingenide} which shows the Artinianness of the modules
$R^i D_I(M,N).$

The last section is devoted to study   the set of associated primes
$\Ass(R^iD_I(M,N)).$ Theorem \ref{wLGLC} shows that if  $M$ is a
finitely generated $R$-module and $N$ is a weakly Laskerian
$R$-module, then  $\Supp(H^i_{\mathfrak{m}}(M,N))$ and
$\Ass(R^iD_{\mathfrak{m}}(M,N))$  are  finite sets for all $i\geq
0.$ Finally, Theorem \ref{spectraland} gives us two interesting
consequences about the finiteness of the sets $\Ass(R^tD_I(M,N))$
(Corollary \ref{C:3.4}) and $\Supp_R(R^tD_I(M,N))$ (Corollary
\ref{C:hqsupphh}).

\bigskip

\section{Some basic properties of generalized ideal transforms}

An $R-$module $N$ is called $I$-torsion if $\Gamma_I(N)\cong N.$  We
have the first following result.

\begin{theorem}\label{DNfree}
Let $M$ be a finitely generated $R$-module and $N$ an $I$-torsion
$R$-module. Then $R^iD_I(M,N)=0$ for all $i\geq 0$.
\end{theorem}
\begin{proof}
We first prove $D_I(M,N)=0$.

Consider the $n$-th injection $\lambda_n: \Hom_R(I^nM, N)\to
\bigoplus \Hom_R(I^nM,N)$   and the homomorphisms
 $\varphi^i_j: \Hom_R(I^iM, N)\to \Hom_R(I^jM,N)$ such that $\varphi^i_j(f_i)=f_i|_{I^jM}$ for all $i\leq j$.

Let $S$ be an $R$-submodule of $\bigoplus \Hom_R(I^nM,N)$ which is
generated by elements $\lambda_j\varphi^i_j(f_i)-\lambda_if_i,
 f_i\in \Hom_R(I^iM,N)$ and $i\leq j.$
 Then
 $$\mathop {\lim }\limits_{\begin{subarray}{c}
   \longrightarrow  \\
  n
\end{subarray}}  \Hom_R(I^nM,N)=\big(\bigoplus \Hom_R(I^nM,N)\big)/S.$$

For any $u\in D_I(M,N)=\mathop {\lim }\limits_{\begin{subarray}{c}
   \longrightarrow  \\
  n
\end{subarray}}  \Hom_R(I^{n}M,N),$ we have $u=\lambda_tf_t+S$, where $f_t\in \Hom_R(I^tM,N)$.

Since $I^tM$ is a finitely generated $R$-module and $N$ is an
$I$-torsion $R$-module, there exists a positive integer $p$ such
that $\varphi^t_{p+t}(f_t)=0$.

It follows from \cite[2.17 (ii)]{Rotman} that $u=0$ and then
$D_I(M,N)=0$.

The proof will be complete if we show $R^iD_I(M,N)=0$ for all $i>0$.

As $N$ is $I$-torsion, there is an injective resolution
$E^{\bullet}$  of $N$ such that each term of the resolution is an
$I$-torsion $R$-module.
 By the above proof, we have $\mathop {\lim }\limits_{\begin{subarray}{c}
   \longrightarrow  \\
  n
\end{subarray}}  \Hom_R(I^{n}M,E^i)=0$ for all $i\geq 0.$
Therefore $R^iD_I(M,N)=0$ for all $i\geq 0$.
\end{proof}

\begin{corollary}\label{235}
Let $M$ be a finitely generated $R$-module and $N$ an $R$-module
such that $D_I(N)=0$. Then $R^iD_I(M,N)=0$ for all  $i\geq 0$.
\end{corollary}
\begin{proof}
We consider the exact sequence
$$0\to \Gamma_I(N)\to N\to D_I(N)\to H_I^1(N)\to 0$$
From the hypothesis,  we have $\Gamma_I(N)\cong N$ that means $N$ is
$I$-torsion. From \ref{DNfree} we have the conclusion.
\end{proof}

The following lemmas  will be used to prove the next propositions.

\begin{lemma} {\rm ({\cite[2.2]{divcof}})}\label{lemma1}
Let $M,N$ be $R$-modules. Then, there is an exact sequence
$$0\to H_I^0(M,N)\to \Hom_R(M,N)\to D_I(M,N)\to H_I^1(M,N)\to\cdots$$
$$\cdots H_I^i(M,N)\to \Ext_R^i(M,N)\to R^iD_I(M,N)\to H_I^{i+1}(M,N)\to\cdots$$
Moreover, if $\pd(M)<\infty$,  then $R^iD_I(M,N)\cong
H_I^{i+1}(M,N)$ for all $i\geq \pd(M)+1$.
\end{lemma}

\begin{lemma} {\rm(\cite[Theorem 1]{KSB})}\label{preDir}
The following conditions on an $R$-module $M$ are equivalent:
\begin{enumerate}
\item[(i)] $M$ admits a resolution by finitely generated
projectives;
\item[(ii)] The functors $\Ext^n_R(M,-)$ preserve direct limits for all $n;$
\item[(iii)] The functors $\Tor^R_n(-,M)$ preserve products for all $n.$
\end{enumerate}
\end{lemma}

The following lemma shows some basic properties of generalized ideal
transforms that we shall use.

\begin{lemma}\label{RiDI}
Let $M$ be a finitely generated $R$-module and $N$ an $R$-module.
Then
\begin{enumerate}
\item[(i)] $D_I(M,N)$ is an $I$-torsion-free $R$-module;
\item[(ii)] $R^iD_I(M,N)\cong R^iD_I(M,N/\Gamma_I(N)) \text{ for all } i\geq
0$;
\item[(iii)] $R^iD_I(M,N)\cong R^iD_I(M,D_I(N)) \text{ for all } i\geq
0$;
\item[(iv)] $D_I(D_I(M,N))\cong D_I(M,N)$;
\item[(v)] $D_I(\Hom_R(M,N))\cong \Hom_R(M,D_I(N))$.
\end{enumerate}
\end{lemma}
\begin{proof}
(i) We have by \ref{preDir}

\begin{tabular}{rl}
$\Gamma_I(D_I(M,N))$&$=\mathop {\lim }\limits_{\begin{subarray}{c}
   \longrightarrow  \\
  n
\end{subarray}} \Hom_R(R/I^n,D_I(M,N))$\\

&$\cong \mathop {\lim }\limits_{\begin{subarray}{c}
   \longrightarrow  \\
  n
\end{subarray}} \mathop {\lim }\limits_{\begin{subarray}{c}
   \longrightarrow  \\
  t
\end{subarray}} \Hom_R(R/I^n,\Hom_R(I^tM,N))$\\

&$\cong \mathop {\lim }\limits_{\begin{subarray}{c}
   \longrightarrow  \\
  n
\end{subarray}} \mathop {\lim }\limits_{\begin{subarray}{c}
   \longrightarrow  \\
  t
\end{subarray}} \Hom_R(R/I^n\otimes I^tM,N)$\\

&$\cong \mathop {\lim }\limits_{\begin{subarray}{c}
   \longrightarrow  \\
  t
\end{subarray}} \mathop {\lim }\limits_{\begin{subarray}{c}
   \longrightarrow  \\
  n
\end{subarray}} \Hom_R(I^tM,\Hom_R(R/I^n,N))$\\

&$\cong \mathop {\lim }\limits_{\begin{subarray}{c}
   \longrightarrow  \\
  t
\end{subarray}} \Hom_R(I^tM,\Gamma_I(N))$\\
&$\cong D_I(M,\Gamma_I(N)).$

\end{tabular}

Since $\Gamma_I(N)$ is an $I$-torsion $R$-module and from
\ref{DNfree}, we get $D_I(M,\Gamma_I(N))=0$. Thus $D_I(M,N)$ is an
$I$-torsion-free $R$-module.

(ii) The short exact sequence
$$0\to \Gamma_I(N)\to N\to N/\Gamma_I(N)\to 0$$
deduces the long exact sequence
$$0\to D_I(M,\Gamma_I(N))\to D_I(M,N)\to D_I(M,N/\Gamma_I(N))\to \cdots$$
$$\cdots\to R^iD_I(M,N)\to R^iD_I(M,N/\Gamma_I(N))\to R^{i+1}D_I(M,\Gamma_I(N))\cdots$$

Then $R^iD_I(M,N)\cong R^iD_I(M,N/\Gamma_I(N))$ for all $i\geq 0,$
as $R^iD_I(M,\Gamma_I(N))=0$.

(iii) The short exact sequence
$$0\to N/\Gamma_I(N)\to D_I(N)\to H^1_I(N)\to 0$$
deduces a long exact sequence
$$0\to D_I(M,N/\Gamma_I(N))\to D_I(M,D_I(N))\to D_I(M,H^1_I(N))\to\cdots$$
$$\to R^iD_I(M,N/\Gamma_I(N))\to R^iD_I(M,D_I(N))\to R^iD_I(M,H^1_I(N))\to$$

As $R^iD_I(M,H^1_I(N))=0$,  $R^iD_I(M,N/\Gamma_I(N))\cong
R^iD_I(M,D_I(N))$ for all $i\geq 0$.

(iv) We have

\begin{tabular}{rl}
$D_I(D_I(M,N))$&$=\mathop {\lim }\limits_{\begin{subarray}{c}
   \longrightarrow  \\
  n
\end{subarray}} \Hom_R(I^n,D_I(M,N))$\\

&$\cong \mathop {\lim }\limits_{\begin{subarray}{c}
   \longrightarrow  \\
  n
\end{subarray}} \mathop {\lim }\limits_{\begin{subarray}{c}
   \longrightarrow  \\
  t
\end{subarray}} \Hom_R(I^n,\Hom_R(I^tM,N))$\\

&$\cong \mathop {\lim }\limits_{\begin{subarray}{c}
   \longrightarrow  \\
  n
\end{subarray}} \mathop {\lim }\limits_{\begin{subarray}{c}
   \longrightarrow  \\
  t
\end{subarray}} \Hom_R(I^n\otimes I^tM,N)$\\

&$\cong \mathop {\lim }\limits_{\begin{subarray}{c}
   \longrightarrow  \\
  t
\end{subarray}} \mathop {\lim }\limits_{\begin{subarray}{c}
   \longrightarrow  \\
  n
\end{subarray}} \Hom_R(I^tM,\Hom_R(I^n,N))$\\
&$\cong \mathop {\lim }\limits_{\begin{subarray}{c}
   \longrightarrow  \\
  t
\end{subarray}} \Hom_R(I^tM,D_I(N))$\\
&$\cong D_I(M,D_I(N))\cong D_I(M,N).$
\end{tabular}

(v) From \ref{preDir} we have

\begin{tabular}{rl}
$D_I(\Hom_R(M,N))$&$=\mathop {\lim }\limits_{\begin{subarray}{c}
   \longrightarrow  \\
  n
\end{subarray}} Hom_R(I^n,\Hom_R(M,N))$\\

&$\cong \mathop {\lim }\limits_{\begin{subarray}{c}
   \longrightarrow  \\
  n
\end{subarray}} \Hom_R(M,\Hom_R(I^n,N))$\\

&$\cong \Hom_R(M,D_I(N))$
\end{tabular}

as required.
\end{proof}

If $f:N\to N^{'}$ is an $R$-module homomorphism such that $\Ker f$
and $\Coker f$ are both $I$-torsion $R$-modules, then
$R^iD_I(N)\cong R^iD_I(N^{'})$ for all $i\geq 0$ (see
\cite{broloc}). We have a similar property in the case of
generalized ideal transforms.
\begin{proposition}\label{2318}
Let $f: N\to N^{'}$ be an $R$-module homomorphism such that $\Ker f$
and $\Coker f$ are both $I$-torsion $R$-modules. Then
$$R^iD_I(M,N)\cong R^iD_I(M,N^{'})$$
for all non-negative integer $i$.
\end{proposition}
\begin{proof}

Two short exact sequences
$$0\to \text{Ker} f \to N \to \text{Im} f\to 0$$
$$0\to \text{Im}f \to N^{'} \to \text{Coker}f \to 0$$
deduce two long exact sequences
$$0\to D_I(M,\text{Ker} f) \to D_I(M,N)\to D_I(M,\text{Im} f)\to R^1D_I(M,\text{Ker}f) \cdots$$
$$0\to D_I(M,\text{Im} f) \to D_I(M,N^{'}) \to D_I(M,\text{Coker}f)\to R^1D_I(M,\text{Im}f) \cdots$$
Since $\Ker f$ and $\Coker f$ are both $I$-torsion $R$-modules,
$R^iD_I(M,\text{Ker}f)=0$ and $R^iD_I(M,\text{Coker}f)=0$ for all
$i\geq 0$. Hence $R^iD_I(M,N)\cong R^iD_I(M,\text{Im}f)$ and
$R^iD_I(M,\text{Im}f)\cong R^iD_I(M,N^{'})$. Finally, we get
$R^iD_I(M,N)\cong R^iD_I(M,N^{'})$ for all $i\geq 0$.
\end{proof}

Let $N_a$ denote the localization of $N$ respect to the
multiplicatively closed subset $S=\{a^{i}\mid i\in\mathbb{N}\}$. We
have the following theorem.

\begin{theorem}\label{DiHom}
Let $M$ be a finitely generated $R$-module and $N$ an $R$-module. Then
\begin{itemize}
\item[(i)] $D_I(\Hom_R(M,N))\cong D_I(M,N);$
\item[(ii)] If $I=aR$ is a principal
ideal of $R,$ then
 $$D_{aR}(M,N)\cong D_{aR}(M,N)_a.$$
\end{itemize}
\end{theorem}
\begin{proof} (i).
The long exact sequence
$$0\to \Gamma_I(M,N)\to \Hom_R(M,N)\to D_I(M,N) \mathop  \to \limits^f H^1_I(M,N)\to \cdots $$
deduces an exact sequence
$$0\to \Gamma_I(M,N)\to \Hom_R(M,N)\to D_I(M,N) \to \text{Im}f\to 0.$$

Note that $\Im f$ is an $R$-submodule of $H^1_I(M,N)$, then $\Im f$ is an $I$-torsion $R$-module.

Since $\Gamma_I(M,N)$ and $\Im f$ are both $I$-torsion $R$-modules, there are isomorphisms
$$D_I(\Hom_R(M,N))\cong D_I(D_I(M,N))\cong D_I(M,N).$$

(ii).  From \ref{RiDI} we have $D_{I}(M,N)\cong D_I(M,D_I(N))$. We
now consider the module $D_I(M,D_I(N)).$ Since $I=aR$ is a principal
ideal, it follows $D_I(N)\cong N_a$. As $I^nM$ is finitely
generated,  we have by \cite[3.83]{Rotman}
$$\Hom_R(I^nM,N\otimes S^{-1}R)\cong S^{-1}R\otimes
\Hom_R(I^nM,N).$$ It follows
$$\mathop {\lim }\limits_{\begin{subarray}{c}
   \longrightarrow  \\
  n
\end{subarray}} \Hom_R(I^nM,N\otimes S^{-1}R)\cong \mathop {\lim }\limits_{\begin{subarray}{c}
   \longrightarrow  \\
  n
\end{subarray}} S^{-1}R\otimes \Hom_R(I^nM,N).$$
Hence $$D_I(M,D_I(N))\cong S^{-1}R\otimes D_I(M,N).$$ Finally, we
get $D_{aR}(M,N)\cong D_{aR}(M,N)_a.$
\end{proof}

If $E$ is an injective $R$-module, then $\Gamma_I(E)$ is also
injective  and $H^1_I(E)=0$. Hence, the short exact sequence
$$0\to\Gamma_I(E)\to E\to D_I(E)\to 0 $$ is split.
It implies that $D_I(E)$ is an injective $R$-module.

\begin{proposition}
Let $M$ be a finitely generated $R$-module, $N$ an $R$-module and
$J^{\bullet}$ an injective resolution of $N.$
 Then $$R^i D_I(M,N)\cong H^i(\Hom_R(M,D_I(J^{\bullet}))).$$
\end{proposition}
\begin{proof}
Combining \ref{DiHom} with \ref{RiDI} yields
$$\begin{array}{rl}
R^iD_I(M,N)&=H^i(D_I(M,J^\bullet))\\
&\cong H^i(D_I(\Hom(M,J^\bullet)))\\
 &\cong
H^i(\Hom(M,D_I(J^\bullet)))
\end{array}$$
as required.
\end{proof}

Next, we study the finiteness of generalized ideal transforms which
relates to generalized local cohomology modules.

\begin{proposition}\label{Connect1}
Let $M,N$ be two finitely generated $R$-modules and $i$ a positive integer. Then $H^i_I(M,N)$ is finitely generated if and only if $R^{i-1}D_I(M,N)$ is finitely generated.
\end{proposition}
\begin{proof}

Since $M,N$ are finitely generated $R$-modules,  $\Ext^i_R(M,N)$ is
also a finitely generated $R$-module for all $i\geq 0$. By
\ref{lemma1} we have the conclusion.
\end{proof}

\begin{theorem}\label{HomFi} Let $M$ be a finitely generated $R$-module and $N$ an
$R$-module. If $t$ is a non-negative integer such that $R^iD_I(M,N)$
is finitely generated for all $i<t,$ then $\Hom_R(R/I, R^tD_I(M,N))$
is a finitely generated $R$-module.
\end{theorem}

\begin{proof}
We use induction on $t$.

Let $t=0$. We have $\Hom_R(R/I, D_I(M,N))=0$, since $D_I(M,N)$ is $I$-torsion-free.

When $t> 0$, from \ref{RiDI}, it follows
$$R^iD_I(M,N)\cong R^iD_I(M,D_I(N))$$
for all $i\geq 0$.

It is sufficient to prove that $\Hom_R(R/I, R^tD_I(M,D_I(N)))$ is finitely generated.

Since $D_I(N)$ is an $I$-torsion-free $R$-module,
 there is a $D_I(N)$-regular element $x\in I.$
Now the short exact sequence
$$0\to D_I(N)\xrightarrow{x} D_I(N)\to\overline{D_I(N)}\to 0,$$
where $\overline{D_I(N)}=D_I(N)/xD_I(N)$ gives rise a long  exact
sequence
$$0\to D_I(M,D_I(N))\xrightarrow{x} D_I(M,D_I(N))\to D_I(M,\overline{D_I(N)})\cdots$$
$$\cdots R^{t-1}D_I(M,D_I(N))\xrightarrow{h}R^{t-1}D_I(M,\overline{D_I(N)})\xrightarrow{k} R^tD_I(M,D_I(N))\cdots.$$
It induces  a short  exact sequence
$$0\to \text{Im} h\to R^{t-1}D_I(M,\overline{D_I(N)})\to \text{Im} k\to 0.$$
As $R^iD_I(M,D_I(N))$ is finitely generated for all $i<t$, $\Im h$ and

\noindent$R^iD_I(M,\overline{D_I(N)})$ are both finitely generated
$R$-modules for all $i<t-1$.

By the inductive hypothesis, $\Hom_R(R/I, R^{t-1}D_I(M,\overline{D_I(N)}))$ is finitely generated. Hence $\Hom_R(R/I, \text{Im} k)$ is finitely generated.

Next, the exact sequence
$$0\to \text{Im} k\to R^tD_I(M,D_I(N))\xrightarrow{x}R^tD_I(M,D_I(N))$$
deduces a long exact sequence
$$0\to \Hom_R(R/I,\Im k)\to \Hom_R(R/I,R^tD_I(M,D_I(N)))\xrightarrow{x}$$
$$\xrightarrow{x}\Hom_R(R/I,R^tD_I(M,D_I(N)))\to\cdots$$
As $x\in I,$  $\Hom_R(R/I, R^tD_I(M,D_I(N)))$ is finitely generated.
\end{proof}

In \cite[2.5]{ChuTang} Tang and Chu proved that if $H_I^r(M, R/\P)$
is Artinian for any $\P\in \Supp(N)$ and $r\geq \pd(M),$ then
$H^i_I(M, N)$ is Artinian for all $i\geq r.$ We show a similar
following proposition.

\begin{proposition}\label{ArGIT}
Let $M,N$ be two finitely generated $R$-modules with
$\pd(M)<\infty$. Assume that $t$ is a positive integer such that
$t>\pd(M).$
\begin{itemize}
 \item[(i)] If $R^tD_I(M,R/\P)$ is Artinian for all $\P\in
\Supp(N),$ then $R^iD_I(M,N)$ is also an Artinian $R$-module for all
$i\geq t.$
 \item[(ii)] If $R^tD_I(M,R/\P)$ and $H_I^t(M, R/\P)$ are Artinian for all $\P\in
\Supp(N),$ then $\Ext^i_R(M,N)$ is also an Artinian $R$-module for
all $i\geq t.$
\end{itemize}
\end{proposition}
\begin{proof} (i). The proof of (i) is similar to that in the proof of
\cite[2.5]{ChuTang}.

(ii). By  \ref{lemma1}  there is an exact sequence
$$0\to H_I^0(M,N)\to \Hom_R(M,N)\to D_I(M,N)\to H_I^1(M,N)\to\cdots$$
$$\cdots H_I^i(M,N)\to \Ext_R^i(M,N)\to R^iD_I(M,N)\to
H_I^{i+1}(M,N)\to\cdots.$$Thus the conclusion follows from (i) and
\cite[2.5]{ChuTang}.
\end{proof}
In the following theorem we study the Artinianness of modules $R^i
D_I(M,N)$ when $N$ is Artinian or finitely generated.

\begin{theorem} \label{T:tinhatingenide}
Let $M$ be a finitely generated $R$-module.
\begin{itemize}
\item[(i)] If  $N$ is an Artinian
$R$-module, then $R^i D_I(M,N)$ is Artinian for all $i\geq 0.$
\item[(ii)] If $N$ is a finitely generated $R$-module such that $p=\pd (M)$ and $d=\dim(N)$ are
finite,
 then $R^{p+d}D_I(M,N)$ is an Artinian $R$-module.
\end{itemize}
\end{theorem}
\begin{proof}
(i). It follows from \cite[2.6]{namont}  that $H^i_I(M,N)$ is
Artinian for all $i\geq 0.$ On the other hand,  $\Ext^i_R(M,N)$ is
Artinian for all $i\geq 0.$ Thus the claim follows from the exact
sequence of \ref{lemma1}.

(ii).  When $d=\dim(N)=0$. It is clear that $N$ is an Artinian
$R$-module. Hence $\Ext^i_R(M,N)$ is Artinian for all $i\geq 0.$

By  \cite[5.1]{bijaco},  $H^i_I(M,N)=0$ for all $i>\pd(M)$
 and $H^{p}_I(M,N)$ is Artinian. Now  we have the exact sequence by \ref{lemma1}
$$\cdots R^{p-1}D_I(M,N)\to H^p_I(M,N)\to \Ext^p_R(M,N)\to R^pD_I(M,N)\to 0.$$
It follows that $R^{p}D_I(M,N)$ is an
Artinian $R$-module.

Let $d=\dim(N)>0$. Since $\Ext^{i}_R(M,N)=0$ for all $i>\pd(M),$
$$R^{p+d}D_I(M,N)\cong H^{p+d+1}_I(M,N)=0.$$  
This finishes the proof.
\end{proof}
\bigskip

\section{Associated primes of the modules $R^iD_I(M,N)$}
\medskip

To study  some properties of associated primes of $R^iD_I(M,N)$ we
recall the concepts of \emph{weakly Laskerian} modules \cite{divass}
and $FSF$ modules \cite{quyont}.
 An $R$-module $M$ is called
\emph{weakly Laskerian} if the set of associated of prime ideals of
any quotient module of $M$ is finite.  An $R$-module $M$ is called a
$FSF$ module if there is a finitely generated submodule $N$ of $M$
such that the support of $M/N$ is a finite set. Note that a module
$M$ is a weakly Laskerian module if and only if $M$ is a $FSF$
module (see \cite[2.5]{bahont}).

\begin{lemma}[\cite{divass}]\label{3.1}
\begin{itemize}

\item[(i)] Let $0\to M^{'}\to M\to M^{''}\to 0$ be a short exact sequence. Then $M$ is weakly Laskerian if and only if $M^{'}$ and $M^{''}$ are both weakly Laskerian.

\item[(ii)] If $M$ is a finitely generated $R$-module and $N$ is a weakly Laskerian, then $\Ext^i_R(M,N)$ and $\Tor^R_i(M,N)$ are weakly Laskerian for all $i\geq 0$.

\item[(iii)] Artinian modules and finitely generated modules are weakly Laskerian modules.
\end{itemize}
\end{lemma}

\begin{proposition}\label{coroDiHom}
Let $M, N$ be  finitely generated $R$-modules. Then $$\Ass(D_I(M,N))=(\Supp(M)\cap \Ass(N))\setminus V(I).$$
\end{proposition}
\begin{proof} It follows from \cite[3.1]{schonc} that $\Ass(D_I(N))=\Ass(N)\setminus V(I)$.
Then we have by \ref{DiHom}

\begin{tabular}{rl}
$\Ass(D_I(M,N))$&$=\Ass(D_I(\Hom_R(M,N)))$\\
& $=\Ass(\Hom_R(M,N))\setminus V(I)$\\
&$=(\Supp(M)\cap \Ass(N))\setminus V(I)$
\end{tabular}

as required.
\end{proof}

It is well-known that, if $M,N$ are finitely generated $R$-modules,
then $H^i_{\mathfrak{m}}(M,N)$ is Artinian for all $i\geq 0$ (see
\cite[2.2]{divonv}). It implies that
$\Supp(H^i_{\mathfrak{m}}(M,N))$ is a finite set. When $N$ is a
weakly Laskerian $R$-module we have the following theorem.

\begin{theorem}\label{wLGLC}
Let $M$ be a finitely generated $R$-module and $N$ a weakly
Laskerian $R$-module. If  $\mathfrak{m}$ is a maximal ideal of $R,$
then $\Supp(H^i_{\mathfrak{m}}(M,N))$ and
$\Ass(R^iD_{\mathfrak{m}}(M,N))$  are  finite sets for all $i\geq
0.$
\end{theorem}
\begin{proof}
As $N$ is weakly Laskerian, there exists a finitely generated
submodule $L$ of $N$ such that $\Supp(N/L)$ is a finite set. Then
the short exact sequence
$$0\to L\to N\to N/L\to 0$$
deduces a long exact sequence
$$\cdots \to H^i_{\mathfrak{m}}(M,L)\overset{f}\to H^i_{\mathfrak{m}}(M,N)\overset{g}\to H^i_{\mathfrak{m}}(M,N/L)\to \cdots$$

Since $H^i_{\mathfrak{m}}(M,L)$ is an Artinian $R$-module,
$\Supp(H^i_{\mathfrak{m}}(M,L))$ is a finite set and $\Im f$ is an
Artinian $R$-module.

Note that $\Supp(\Im g)$ is a finite set because $$\Supp (\Im g)\subset \Supp(H^i_{\mathfrak{m}}(M,N/L))\subset \Supp(N/L).$$

From the long exact sequence, we obtain a short exact sequence
$$0\to \text{Im}f\to H^i_{\mathfrak{m}}(M,N)\to \text{Im}g\to 0$$
which implies $\Supp(H^i_{\mathfrak{m}}(M,N))= \Supp(\Im f)\cup
\Supp(\Im g)$. Thus $\Supp(H^i_{\mathfrak{m}}(M,N))$ is a finite set
and then
 $\Ass(H^i_\mathfrak{m}(M,N))$ is a finite set.
We now consider the exact sequence
$$0\to \Gamma_\mathfrak{m}(M,N)\to \Hom_R(M,N)\to D_\mathfrak{m}(M,N)\to\cdots$$
$$\cdots\to \Ext^i_R(M,N)\to R^iD_\mathfrak{m}(M,N)\to H^{i+1}_\mathfrak{m}(M,N)\to\cdots $$
Since $\Ass(\Ext^i_R(M,N))$ is a finite set, we have the conclusion.
\end{proof}

\begin{proposition}\label{hoaass3.1}
Let  $M$ be a finitely generated module and $N$ a weakly Laskerian
module over a   local ring $(R, \frak{m}).$ If $\dim (N)\leq 2$,
then $R^iD_I(M,N)$ is weakly Laskerian for all $i\geq 0.$
\end{proposition}
\begin{proof}
From \ref{lemma1} there is an exact sequence
$$0\to H_I^0(M,N)\to \Hom_R(M,N)\to D_I(M,N)\to H_I^1(M,N)\to\cdots$$
$$\cdots H_I^i(M,N)\to \Ext_R^i(M,N)\to R^iD_I(M,N)\to H_I^{i+1}(M,N)\to\cdots$$
If $dim(N)\leq 2$, then $H^i_I(M,N)$ is weakly Laskerian for all
$i\geq 0$ by \cite[3.1]{Hoang}. Note that $\Ext^i_R(M,N)$ is weakly
Laskerian for all $i\geq 0$. Therefore $R^iD_I(M,N)$ is weakly
Laskerian for all $i\geq 0.$
\end{proof}

\begin{theorem}\label{spectraland} Let $M$ be a finitely generated $R$-module and $N$  an
$R$-module. Then
\begin{itemize}
 \item[(i)] There is a Grothendieck spectral sequence
$$E_2^{pq}=\Ext^p_R(M,R^qD_I(N))\underset{p}{\Longrightarrow} R^{p+q}D_I(M,N);$$
\item[(ii)] $\Ass_R (R^t D_I (M,N)) \subseteq (\bigcup\limits_{i = 1}^t \Ass_R
(E_{t+2}^{i,t-i}))\bigcup  \Ass_R(\Hom(M,R^{t}D_I (N)));$
\item[(iii)] $\Supp_R (R^t D_I (M,N)) \subseteq \bigcup\limits_{i = 0}^t {\Supp_R (\Ext_R^i (M,R^{t-i}D_I
(N)))}.$
\end{itemize}
\end{theorem}
\begin{proof} (i). Let us consider functors $F(-) = \Hom_R(M,-)$ and $G(-) =
D_I(-).$ The functor $F(-)$ is left exact.  For any injective module
$E,$  $G(E)$ is also an injective module and then  is right
$F$-acyclic. On the other hand,  there is a natural equivalence  by
\ref{DiHom}, $D_I(M,-)\cong \Hom_R(M,D_I(-))$. Thus from
\cite[11.38]{Rotman} we have the Grothendieck spectral sequence
$$E_2^{pq}=\Ext^p_R(M,R^qD_I(N))\underset{p}{\Longrightarrow}
R^{p+q}D_I(M,N).$$

(ii). From the spectral of (i) there is a finite filtration $\Phi$
of $R^{p+q}D_I(M,N)$ with
$$0=\Phi^{t+1}H^t\subset \Phi^tH^t\subset\ldots\subset\Phi^1H^t\subset\Phi^0H^t=R^t D_I(M,N)$$
and
$$E^{i,t-i}_\infty=\Phi^iH^t/\Phi^{i+1}H^t,\ {\rm where}\ t=p+q,\  0\leq i\leq t.$$
Exact sequences for all $0\leq i\leq t$
$$0\to \Phi^{i+1}H^t\to \Phi^iH^t\to E^{i,t-i}_\infty\to 0$$
gives $$\Ass(\Phi^iH^t)\subseteq \Ass(\Phi^{i+1}H^t) \bigcup
\Ass(E^{i,t-i}_\infty).$$ We may  integrate this for $i=0, 1, ...,
t$ to conclude that $$\Ass(R^t D_I(M,N))\subseteq \bigcup\limits_{i
= 0}^t \Ass_R (E^{i,t-i}_\infty).$$ We now consider homomorphisms of
the spectral
$$E_{t+2}^{i-t-2,2t-i+1}\longrightarrow E_{t+2}^{i,t-i}\longrightarrow
E_{t+2}^{i+t+2,-i-1}.$$Note that $E_{t+2}^{i-t-2,2t-i+1} =
E_{t+2}^{i+t+2,-i-1} = 0$ for $i=0, 1, ..., t.$ It follows
$$E_{t+2}^{i,t-i} = E_{t+3}^{i,t-i} = ... = E_{\infty}^{i,t-i}.$$In
particular,$$E_{\infty}^{0,t} = ... = E_{t+3}^{0,t} =
E_{t+2}^{0,t}\subseteq E_{t+1}^{0,t}\subseteq ... \subseteq
E_2^{0,t}.$$ Therefore $$\Ass_R (R^t D_I (M,N)) \subseteq
(\bigcup\limits_{i = 1}^t \Ass_R (E_{t+2}^{i,t-i}))\bigcup
\Ass_R(\Hom(M,R^{t}D_I (N))).$$

(iii). Analysis similar to that in the proof of (ii) shows that
 $$\Supp(R^t D_I(M,N))\subseteq \bigcup\limits_{i
= 0}^t \Supp_R (E^{i,t-i}_\infty)$$and
$$E_{t+2}^{i,t-i} = E_{t+3}^{i,t-i} = ... =
E_{\infty}^{i,t-i}.$$Thus $E_{\infty}^{i,t-i}$ is a subquotient of
$E_{2}^{i,t-i}$ and then $$\Supp_R (E^{i,t-i}_\infty)\subseteq
\Supp_R (E_{2}^{i,t-i}) = \Supp_R (\Ext_R^i (M,R^{t-i}D_I
(N)))).$$This finishes the proof.
\end{proof}

\begin{corollary}\label{C:3.4}
Let $M$ be a finitely generated $R$-module, $N$ a weakly Laskerian
$R$-module and $t$ a non-negative integer. If $R^iD_I(N)$ is weakly
Laskerian  for all  $i< t,$ then $\Ass(R^tD_I(M,N))$ is a finite
set.
\end{corollary}
\begin{proof}
We have by \ref{spectraland}
 $$\Ass_R (R^t D_I (M,N)) \subseteq (\bigcup\limits_{i = 1}^t
\Ass_R (E_{t+2}^{i,t-i}))\bigcup \Ass_R(\Hom(M,R^{t}D_I (N))).$$ As
$R^iD_I(N)$ is weakly Laskerian  for all  $i< t,$
$\Ext^i_R(M,R^{t-i}D_I(N))$ is also weakly Laskerian  for all  $1\leq i<
t.$ Since $E_{t+2}^{i,t-i}$ is a subquotient of $E_{2}^{i,t-i} =
\Ext^i_R(M,R^{t-i}D_I(N)),$  $E_{t+2}^{i,t-i}$ is weakly Laskerian
for all  $1\leq i< t.$ It follows that $\bigcup\limits_{i = 1}^t \Ass_R
(E_{t+2}^{i,t-i})$ is finite. Note that  $R^{i}D_I (N)\cong
H^{i+1}_I(N)$ for $i>0.$ Thus  $H^{i}_I(N)$ is weakly Laskerian  for
all  $i< t+1.$ By \cite[2.2]{brofin}, $\Ass_R(H^{t+1}_I(N))$ is
finite and then $\Ass_R(\Hom(M,R^{t}D_I (N)))$ is finite. The proof
is complete.
\end{proof}

\begin{corollary}\label{C:hqsupphh}
Let $M$ be a finitely generated $R$-module, $N$ an $R$-module and
$t$ a non-negative integer. If $\Supp_R(R^iD_I(N))$ is a finite set
for all $i\leq t,$ then $\Supp_R(R^tD_I(M,N))$ is also a finite set.
\end{corollary}
\begin{proof}
It follows from \ref{spectraland} that \begin{align*} \Supp_R (R^t
D_I (M,N))& \subseteq \bigcup\limits_{i = 0}^t {\Supp_R (\Ext_R^i
(M,R^{t-i}D_I (N)))}\\
&\subseteq\bigcup\limits_{i = 0}^t {\Supp_R (R^{t-i}D_I
(N)).}\end{align*} By the hypothesis we have the conclusion.
\end{proof}

\end{document}